\documentclass[fleqn,11pt,a4paper]{article}

\usepackage{a4wide}
\usepackage{graphicx}
\usepackage{amsfonts}
\usepackage{amsmath}
\usepackage{amssymb}
\usepackage{amsthm}
\usepackage{enumitem}
\setenumerate{itemsep=0pt,topsep=0pt,labelindent=0em,leftmargin=10mm}
\usepackage{xcolor}
\usepackage{float}

\newtheorem{theorem}{Theorem}
\newtheorem{lemma}[theorem]{Lemma}

\newtheorem{proposition}[theorem]{Proposition}
 \theoremstyle{definition}
\newtheorem{definition}[theorem]{Definition}
\newtheorem{remark}[theorem]{Remark}

\newfloat{algorithm}{p}

\newcommand{\floor}[1]{\lfloor #1 \rfloor }

\begin{document}

\title{On absolutely normal numbers and their discrepancy estimate}

\author{Ver\'onica Becher\qquad Adrian-Maria Scheerer \qquad Theodore Slaman}

\date{February 2017}

%
%
%


\maketitle

\vspace*{-8cm} 
This paper has been superseded by:

``On the construction of absolutely normal numbers''

Christoph Aistleitner, Ver\'onica Becher, Adrian-Maria Scheerer and Theodore Slaman

July 2017, 	arXiv:1707.02628

http://arxiv.org/abs/1707.02628
\vspace*{6cm}

\begin{abstract}
  We construct the base~$2$ expansion of an absolutely normal real number~$x$ so that for every
  integer~$b$ greater than or equal to~$2$ the discrepancy modulo~$1$ of the sequence
  $(b^0 x, b^1x, b^2 x , \ldots)$ is essentially the same as that realized by almost all real
  numbers.
\end{abstract}
\bigskip

For a real number $x$, we write     $\{x\} = x-\floor{x}$ to denote the fractional part of $x$.
For a sequence  $(x_j)_{j\geq 1}$ of real numbers in the unit interval,
the discrepancy of the $N$ first elements~is
\[
D_N((x_j)_{j\geq 1})= \sup_{0\leq u< v\leq 1} \left| \frac{\#\{ j : 1\leq j \leq N  \text{ and } u\leq x_j < v\}}{N} - (v-u)\ \right|.
\]
In this note we prove the following.

\begin{theorem}\label{thm:main}
There is an  algorithm that computes  a real number~$x$  such that 
for each integer~$b$ greater than or equal to~$2$,
\begin{align*}
\limsup_{N \to \infty} \frac{D_N( \{b^j x\}_{j\geq 0}) \sqrt{N}}{\sqrt{\log \log N}} < 3 C_b,
\end{align*}
where  
\begin{equation*}
 C_b= 166 + 664/(\sqrt{b}-1)  \text{  is  Philipp's constant}.
\end{equation*}
The algorithm computes the first $n$ digits of the expansion of~$x$
in base $2$  after performing triple-exponential in $n$ 
mathematical operations.
\end{theorem} 

It is well known that for almost all  real numbers $x$ and for all integers $b$
greater than or equal to $2$,  the sequence $\{b^j x\}_{j\geq 0}$
is uniformly distributed in the unit interval, which means that its discrepancy tends to $0$
as $N$ goes to infinity.
 In~\cite{GalGal1964}, G\'al and G\'al  proved
that there is a constant $C$ such that for almost all real numbers~$x$,
\[
\limsup_{N\to \infty} \frac{D_N( \{2^j x\}_{j\geq 0}) \sqrt{N}}{\sqrt{\log \log N}} < C .
\]
Philipp~\cite{Philipp1975}  bounded  the existential constant $C$
and extended this result for lacunary sequences.
He proved  that given a sequence of positive integers $(n_j)_{j \geq 1}$ 
such that $n_{j+1}/n_j\geq \theta$ for some real number $\theta>1$, then for almost all real numbers~$x$ the sequence $\{n_j x\}_{j\geq 1}$ satisfies
\[
\limsup_{N\to \infty} \frac{D_N( \{n_j x\}_{j\geq 1} ) \sqrt{N}}{\sqrt{\log \log N}} < 
166 + 664/(\sqrt{\theta}-1).
\]
Finally,  Fukuyama~\cite{Fukuyama2008} explicitly determined, for any real~$\theta>1$, the constant~$C'_\theta$ (see \cite[Corollary]{Fukuyama2008})
such that  for almost all real numbers~$x$, 
\[
\limsup_{N\to \infty} \frac{D_N( \{\theta^j x\}_{j\geq 0} ) \sqrt{N}}{\sqrt{\log \log N}} =
C'_\theta.
\]
For instance,  in case~$\theta$ is an integer greater than or equal to $2$,
\[
C'_\theta=\left\{
\begin{array}{ll}\medskip
\sqrt{84}/9, & \text{ if } \theta=2
\\\medskip
\sqrt{2 (\theta+1)/(\theta-1)}/2, &  \text{ if  $\theta$ is odd}
\\\medskip
\sqrt{2(\theta+1)\theta(\theta-2)/(\theta-1)^3}/2, & \text{ if } \theta\geq 4 \text{ is even}.  
\end{array}
\right.
\]

The proof of Theorem~\ref{thm:main} is based on the explicit construction 
of a set of full Lebesgue measure
given by Philipp in~\cite{Philipp1975}, 
which, in turn,  follows from that in~\cite{GalGal1964}.
Unfortunately we do not know  an explicit construction 
of a set with full Lebesgue measure achieving the  constants  proved 
by Fukuyama~\cite{Fukuyama2008}.
If one could give such an  explicit construction one could obtain 
a version of  Theorem~\ref{thm:main} with the constant $3C_b$ replaced by $C'_b$.

The  algorithm  stated  in Theorem~\ref{thm:main} achieves a lower 
 discrepancy bound   than that  in Levin's work~\cite{Levin1979}.  
Given a countable set $L$ of  positive real numbers greater than~$1$, Levin 
constructs a real number $x$ such that 
for every $\theta$ in $L$ there is a constant $C''_\theta$ such that
\[
D_N(\{\theta^j x\}_{j\geq 0}) < C''_\theta \frac{(\log N)^3}{\sqrt{N}}.
\]
The recent   analysis  in~\cite{scheerer} reports no  constructions with smaller discrepancy.

For $L = \{2,3, \ldots\}$, Levin's construction  produces a computable sequence of real 
numbers that converge to an absolutely normal number~\cite{AlvarezBecher2016}. 
To compute the $n$-th term it requires  double-exponential in $n$ many operations including
 trigonometric operations. 
In contrast, the algorithm presented in Theorem~\ref{thm:main}
 is based just on  discrete mathematics and 
  yields the expansion of the computed  number by outputting one digit after the other.
Unfortunately, to compute the first $n$ digits it performs triple-exponential in $n$ many operations.
Thus, the question raised in~\cite{poly} remains open : 
\begin{quote}
Is there an absolutely normal number computable in polynomial time having a 
nearly optimal discrepancy of normality ?
\end{quote}

Finally we comment that it  is possible to prove a version of Theorem~\ref{thm:main} 
replacing the set of integer bases  by any countable set of  computable 
 real numbers greater than $1$.
The proof would remain  essentially the same except that one needs a suitable version of  
 Lemma~\ref{lemma:bound}.

\section{Primary definitions and results}

We use some tools from~\cite{GalGal1964} and~\cite{Philipp1975}.
For non-negative integers $M$ and $N$, 
for a sequence of real numbers $(x_j)_{j\geq 1}$
and  for real numbers $\alpha_1, \alpha_2$ 
such that $0\leq \alpha_1<\alpha_2\leq 1$, we define
\begin{align*}
F(M,N,\alpha_1, \alpha_2, (x_j)_{j\geq 1}) =
\big|& \#\{j: M\leq  j < M+N:  \alpha_1\leq x_j< \alpha_2 \} -  (\alpha_2-\alpha_1) N \big|.
\end{align*}
\pagebreak
We write $\mu$ to denote Lebesgue measure.

\begin{lemma}[\protect{\cite[Lemma 8]{bfp2007}, 
adapted from Hardy and Wright~\cite[Theorem 148]{HarWri08}}]\label{lemma:turing}
Let  $b$ be an integer greater than or equal to $2$. 
Let  $m$ and $N$ be  positive integers and let  $\varepsilon$ be a real such that
$6/\lfloor{N/m\rfloor}\leq \varepsilon\leq 1/b^m$. 
Then, for any non-negative integer $M$ and for any integer $a$ such that $0\leq a< b^m$,
\[
\mu \{x\in (0,1): |
F(M,N, a b^{-m}, (a+1) b^{-m}, \{b^j x\}_{j\geq 0})| >\varepsilon N \}
\]
is less than 
$2 b^{2m-2}  m\  e^{-\varepsilon^2  N b^m/(6m)}.$
\end{lemma}


The next lemma is  similar to Lemma~\ref{lemma:turing}  but it considers
dyadic intervals instead of $b$-adic intervals.

\begin{lemma}\label{lemma:bound}
Let $b$ be an integer greater than or equal to $2$, let $k$ and $N$ be positive integers
and let  $\varepsilon$ be a real.
Then,  for any pair of integers  $M$ and $a$ such that $M\geq 0$ and 
$0\leq a<2^k$,
\[
\mu\left\{x\in (0,1): 
F(M,N,a 2^{-k},(a+1) 2^{-k} ,\{b^j x\}_{j\geq 0})  \geq \varepsilon N\right\}
\]
is less than 
$9\cdot 2^{2(k+2)} (k+2) e^{-\varepsilon^2 N b^{k+2} /(6 (k +2))}$.
\end{lemma}

\begin{remark}
In~\cite{Philipp1975}, Philipp  proves a proposition more general than Lemma~\ref{lemma:bound}.
His result yields the same order of magnitude but does not make explicit the underlying constant
while Lemma~\ref{lemma:bound}~does.
\end{remark}

Clearly,
for arbitrary reals $\alpha_1, \alpha_2$ such that $0\leq \alpha_1< \alpha_2\leq 1$, 
for any sequence $(x_j)_{j\geq 1}$
and for any non-negative  integers $M$, $N$ and  $k$,
\[
|F(0,N,\alpha_1,\alpha_2,(x_j)_{j\geq 1})|\leq N/2^{k-1} +
\sum_{m=1}^k \max_{ 0\leq a < 2^m} |F(0,N,  a2^{-m},  (a+1) 2^{-m},(x_j)_{j\geq 1})|.
\]

\begin{lemma}[\protect{\cite[Lemma 4]{Philipp1975}}, adapted from \protect{\cite[Lemma 3.10]{GalGal1964}}]
Let $b$ be an integer greater than or equal to~$2$,
let $N$ be a positive integer and let $n$ be such that 
$2^n\leq N< 2^{n+1}$.
Then, there are integers $m_1,\ldots, m_n$  
with  $0\leq m_\ell \leq 2^{n-\ell}-1$ for  $\ell=1, \ldots, n$,
such that for any positive integer $h$ and any $a$,
with $0\leq a<2^h$,
\begin{align*}
F(0,N,a 2^{-h}, (a+1) 2^{-h},\{b^j x\}_{j\geq 0}) & \leq  
 N^{1/3}+ F(0,2^n,a 2^{-h}, (a+1) 2^{-h}, \{b^j x\}_{j\geq 0}) 
\\
+&\sum_{\ell=n/2}^{n } F(2^n+ m_\ell 2^\ell, 2^{\ell-1},a 2^{-h}, (a+1) 2^{-h}, \{b^j x\}_{j\geq 0}).
\end{align*}
 \end{lemma}

Let $\eta$ and $\delta$ be  positive reals. 
For each integer $b$ greater than or equal to $2$ 
and for  each positive integer $N$ let 
\begin{align*}
\tilde C_b &=1/2 + 2/(\sqrt{b}-1),
\\
\phi(N)&=2(1+2\delta) \tilde C_b (N\log\log N)^{1/2},
\\
T(N)&=\lfloor \log N/\log 4\rfloor+1.
\end{align*}
For integers $b, n, a, h, \ell$ and $m$ such that 
\begin{align*}
b\geq 2, \ n\geq 1, \  0\leq a< 2^{T(2^n)}, \ 
1\leq h\leq T(2^n), \ n/2\leq \ell\leq n, \text{ and } 1\leq m\leq 2^{n/2},
\end{align*}
define the following sets
\begin{align*}
G(b,n, a, h) =& \{x\in(0,1): F(0,2^n, \alpha_1,\alpha_2,\{b^j x\}_{j\geq 0})\geq  2^{-h/8} \phi(2^n)\}, \\
\text{where }&\alpha_1=a 2^{-(h+1)}, \  \alpha_2= (a+1) 2^{-(h+1)}, \text{ if } 1\leq h<T(2^n);
\\
\text{and }  &\alpha_1=a 2^{-T(2^n)}, \ \alpha_2= (a+1) 2^{-T(2^n)}, \text{ if } h=T(2^n).
\\
H(b,n,a, h, \ell, m)=& \{x\in(0,1): F(2^n+ m2^{\ell},2^{\ell -1},\beta_1,\beta_2,\{b^j x\}_{j\geq 0})\geq  2^{-h/8} 2^{(\ell-n-3)/6}\phi(2^n)\},
\\
\text{where }&\beta_1=a 2^{-(h+1)}, \ \beta_2= (a+1) 2^{-(h+1)}, \text{ if } 1\leq h<T(2^{\ell-1});
\\
\text{and }  &\beta_1=a 2^{-T(2^{\ell-1})}, \ \beta_2= (a+1) 2^{-T(2^{\ell-1})}, \text{ if } h=T(2^{\ell-1}).
\\
G_{b,n}=&
\bigcup_{h=1}^{T(2^n)}
\bigcup_{a=0}^{2^{h}-1} G(b,n,a,h),\\
H_{b,n}=&
\bigcup_{h=1}^{T(2^n)}
\bigcup_{a=0}^{2^{h}-1}
\bigcup_{\ell=n/2}^{n}
\bigcup_{m=1}^{2^{n-\ell}}
H(b,n,a, h, \ell, m).
\end{align*}

\begin{lemma}\label{lemma:main}
Let $\eta$ and $ \delta$ be positive real numbers. 
For each  $n>e^{6/(\delta \log 2)}$ and for every  $b \geq 2$,
\[
\mu(G_{b,n})= n^{-1-4\delta}, \quad \mu(H_{b,n})=2 n^{-1-3\delta},
\]
and there is $n_0=n_0(\eta, \delta)$ such that
\[
 \mu\left(\bigcup_{n\geq n_0}(G_{b,n}\cup H_{b,n})\right)<\eta
\] 
and such that for every real  $x$ outside $\bigcup_{n\geq n_0}(G_{b,n}\cup H_{b,n})$,
\[
\limsup_{N\to \infty } \frac{D_N( \{b^n x\}_{n\geq 0}) \sqrt{N}}{\sqrt{\log \log N}} < (1+4\delta) C_b,
\]
where $C_b$ is Philipp's constant, $C_b= 166 + 664/(\sqrt{b}-1)$.
\end{lemma}         
    
\begin{proof}
To bound $\mu G_{b,n}$ we apply twice Lemma~\ref{lemma:bound}, first 
with $N=2^n$, $k=(h+1)$  and $\varepsilon= 2^{-T(2^n)/8} \phi(2^n) 2^{-n}$,
and then with $N=2^n$, $k=T(2^n)$  and $\varepsilon= 2^{-T(2^n)/8} \phi(2^n) 2^{-n}$.
We write $\exp(x)$ to denote $e^{x}$ and we write $T$ instead of  $T(2^n)$.
Assuming~$n\geq 10$, 
\begin{align*}
\mu G_{b,n} \leq  & \ \mu \left(\bigcup_{a=0}^{2^{T}-1} G(b,n,a,T)\right) +
\sum_{h=1}^{T-1 } \mu \left(\bigcup_{a=0}^{2^{h}-1} G(b,n,a,h)\right)
\\
\leq  & \
\sum_{h=1}^{T-1 }  
2^h  9 \cdot  2^{2(h+1)} (h+3)  \exp\left(-2^{-h/4} \phi^2(2^n)2^{-n}  b^{h+1} \frac{b^2}{6 (h+3)} \right) 
\\
&+ 2^T  9\cdot  2^{2(T+2)} (T+2)  \exp\left(-2^{-T/4} \phi^2(2^n) 2^{-n}  b^{T} \frac{b^2}{6 (T+2)}\right) 
\\
\leq & \ 9 \cdot 2^{3T+5}    (T+2) 
\exp\left(-2^{-T/4} \log\log(2^n) 4   (1+\ 2\delta)^2 b^{T+2}\frac{1}{6 (T+2)} \tilde C^2_b\right)
\\
\leq& \  n^{-(1+4\delta)}.
\end{align*}
To bound $\mu H_{b,n}$ we apply twice Lemma~\ref{lemma:bound}
first letting 
 $N=2^{\ell-1}$, $k=(h+1)$  and \mbox{$\varepsilon= 2^{-h/8} \phi(2^n) 2^{-n}$,}
and then letting
$N=2^{\ell-1}$, $k=T(2^{\ell-1})$  and $\varepsilon=   2^{-T(2^{\ell-1})/8} \phi(2^n) 2^{-n}$.
Again we write $T$ instead of $T(2^{n})$. 
Assuming $\log\log (2^n)\geq 8/\delta^2$,
\begin{align*}
\mu& H_{b,n} = \
\mu \left(\bigcup_{h=1}^{T}
\bigcup_{a=0}^{2^{h}-1}
\bigcup_{\ell=n/2}^{n}
\bigcup_{m=1}^{2^{n-\ell}}
H(b,n,a, h, \ell, m) \right)
\\
\leq&\
 \sum_{\ell=n/2}^{n} 2^{n-\ell} \sum_{h=1}^{T-1}  
9\  2^{3h+6} (h+3) 
\exp\left(-2^{-h/4}   b^{h+3} 2^{2(n-\ell)/3 } 
\log\log(2^n) (1+\delta)^2 \frac{4}{6 (h +3)} \tilde C^2_b   \right)
\\
&\ +
 \sum_{\ell=n/2}^{n} 2^{n-\ell}
9\  2^{3T+4} (T+2) 
\exp\left(- 2^{-T/4}   b^{T+2}   2^{2(n-\ell)/3}   \log\log(2^n) (1+\delta)^2   \frac{4}{6 (T +2)}  \tilde C^2_b \right)
\\
\leq&\
 \sum_{\ell=n/2}^{n} 2^{n-\ell} 
\exp\left(-2^{-1/4}    2^{2(n-\ell)/3 } 
\log\log(2^n) (1+4\delta)  \frac{b^{4}}{24} \right)
\sum_{h=1}^{T-1}  2^{-h}
\\
& +
 \sum_{\ell=n/2}^{n} 2^{n-\ell}
9\  2^{3T+4} (T+2) 
\exp\left(- 2^{-T/4}   b^{T+2}  2^{2(n-\ell)/3} 
  \log\log(2^n) (1+\delta)^2  \frac{1}{6 (T +2)}   \right)
\\
\leq&\
 \exp\left(-2^{-1/4}     \log\log(2^n) (1+3\delta) \frac{b^{4}}{24} \right)
\sum_{\ell=n/2}^n 2^{n/2-\ell-1}  
\\
& 
+
 \exp\left(-   2^{-T/4}   b^{T+2}       \log\log(2^n) (1+3\delta)  \frac{1}{6 (T +2) }  \right) 
\sum_{\ell=n/2}^n 2^{n/2-\ell-1}  
\\
\leq&\ 2\ n^{-(1+3\delta)}.
\end{align*}
Thus, there is  $n_0$ such that for every integer $b$ greater than or equal to $2$,
\[
\mu\left(\bigcup_{n\geq n_0}(G_{b,n}\cup H_{b,n})\right)<
\sum_{n\geq n_0} \left(n^{-1-4\delta}+2n^{-1-3\delta}\right)<\eta.
\]
It follows from Philipp's proof  of \cite[Theorem 1]{Philipp1975} that 
for every real $x$ outside \mbox{$\bigcup_{n\geq n_0}(G_{b,n}\cup H_{b,n})$,}
\[
\limsup_{N\to \infty } \frac{D_N( \{b^j x\}_{j\geq 0}) \sqrt{N}}{\sqrt{\log \log N}} < (1 + 4\delta) C_b,
\]\nopagebreak
where $C_b= 166 + 664/(\sqrt{b}-1)$.\qedhere
\end{proof}

\section{Proof of  Theorem~\ref{thm:main}}


We give an algorithm to compute a real  outside the set 
$\bigcup_{b\geq 2}\bigcup_{n\geq n_0}\left(G_{b,n}\cup H_{b,n}\right)$.
The technique is similar to that used 
in  the computable reformulation of Sierpinski's construction  given in  \cite{BecherFigueira2002}.

The next definition introduces finite approximations to this set.
Recall that by Lemma~\ref{lemma:main}, for every  integer $b\geq 2$, provided $\delta\geq 1/2$
and $n_0=n_0(\eta,\delta)\geq  e^{6/(\delta^2\log 2)}$,
\[
\mu\left(\bigcup_{n\geq n_0}\left(G_{b,n}\cup H_{b,n}\right)\right)
\leq \sum_{n\geq n_0} n^{-(1+4\delta)} + 2 n^{-(1+3\delta)}
\leq \sum_{n\geq n_0} n^{-2}<\eta.
\]

\begin{definition}\label{def:main}
Fix $\delta=1/2$ and fix $\eta\leq 1/8$.
For each integer  $b\geq 2$, let  $z_b$
be the least integer greater than $e^{6 /(\delta\log 2)}=e^{12/\log 2} $ such that  
\begin{align*}
\sum_{k=z_b}^\infty \frac{1}{k^2}< \frac{\eta}{2^b}.
\end{align*}
We define
\begin{align*}
\Delta=&\bigcup_{b= 2}^\infty\bigcup_{m= z_b}^\infty (G_{b,m}\cup H_{b,m}),
\\
s=& \sum_{b=2}^\infty \sum_{k=z_b}^\infty \frac{1}{k^2}.
\end{align*}
Observe that  $\mu(\Delta)< s < \eta$.
\medskip

For each $n$, let
\begin{align*}
b_n=& \max(2,\floor{\log_2 n}),
\\
\Delta_n=& \bigcup_{b=2}^{b_n}  \bigcup_{m=z_b}^{n} ( G_{b,m}\cup H_{b,m}),
\\
s_n=&   \sum_{b=2}^{b_n} \sum_{k=z_b}^{n} \frac{1}{k^2},
\\
r_n =&s-s_n =
     \sum_{b=2}^{b_{n}} \sum_{k=\max(n+1,z_b)}^{\infty} \frac{1}{k^2}
    + \sum_{b=b_{n}+1}^\infty \sum_{k=z_b}^\infty \frac{1}{k^2},
\\
p_n=&2^{2n+2}.
\end{align*}
\end{definition}

The next propositions follow immediately from these definitions.

\begin{proposition} \label{prop:rm} 
 For every $n$, $\mu \left( \Delta -\Delta_{n}\right) \leq r_n.$
\end{proposition}

\begin{proposition} \label{prop:rmrq} 
For every $n$ and $q$ such that
$n\leq q$, $\mu \left( \Delta _{q}-\Delta _{n}\right) \leq r_{n}-r_{q}.$
\end{proposition}

\begin{proposition}\label{lemab} 
For any interval $I$ and any  $n$,
$\mu \left( \Delta \cap I\right) \leq \mu \left( \Delta _{n}\cap I\right) +r_n.$
\end{proposition}

The proof of Theorem~\ref{thm:main} follows from the next lemma.
\begin{lemma}\label{lemma:algorithm}
There is a computable sequence of nested dyadic intervals $I_0,I_1, I_2, \ldots$ such that 
for each $n$, 
$\mu I_n=2^{-n}$ and $\mu(\Delta\cap I_n) < 2^{-n}$.
\end{lemma}

\begin{proof}
Proposition~\ref{lemab} establishes, for any interval $I$ and any $m$,
\[
\mu \left( \Delta \cap I\right) <
\mu \left( \Delta_m \cap I\right) +r_m.
\]
Then, to prove the lemma it suffices to give  
a computable sequence of nested dyadic intervals $I_0, I_1, I_2, \ldots$ such that 
for each $n$, 
$\mu I_n=2^{-n}$ and 
$\mu \left( \Delta_{p_{n}}\cap I_{n}\right) +r_{p_{n}}<{2^{-n}}$.
We establish 
\[
p_n= 2^{2 n+2}.
\] 
This value of $p_n$
 is large enough so
that the error $r_{p_n}$ is sufficiently small to guarantee that
even if all the intervals  in $\Delta-\Delta_{p_n}$
 fall in the half of $I_n$ that will be chosen as $I_{n+1}$,
 $I_{n+1}$ will not be completely covered by $\Delta$. 
We define the $I_0, I_1, \ldots$  inductively.
\medskip

{\em Base case, $n=0$}.
Let $I_0= \left[0,1\right)$.
We need to check that 
 $\mu \left( \Delta_{p_0}\cap I_{0}\right) +r_{p_0}< 2^{0}$.
Since $p_0=2^{2\cdot 0+2}= 4$,  $b_{p_0}=2$ and $z_b=  2^2/\eta\geq 2^5 $,
\[
\Delta_{p_0}= \bigcup_{b=2}^{b_{p_0}}  \bigcup_{n=z_b}^{p_0}  (G_{b,n}\cup H_{b,n}) =\emptyset.
\]
Since $I_0=(0,1)$  and $\Delta_{p_0}=\emptyset $, $\Delta_{p_0}\cap I_{0}=\emptyset $.
Then,
\begin{align*}
r_{p_0}=  s =  \sum_{b=2}^\infty \sum_{k=z_b}^\infty \frac{1}{k^2}.
\end{align*}
We conclude  $\mu \left( \Delta_{p_0}\cap I_{0}\right) +r_{p_0} = 0 + s < \eta< 1$.
\medskip

{\em Inductive case, $n>0$}.
Assume that  for  each $m=0,1, \ldots, n-1$,
\[
\mu \big( \Delta _{p_{m}}\cap I_{m}\big) +r_{p_{m}}
<\frac{1}{2^{m}}\Big( \eta
+\sum\limits_{j=1}^{m}2^{j-1}\cdot r_{p_{j}}\Big),
\]
where $ p_{m}= 2^{2m+2}$. 
Note that for $m=0$, $\sum_{j=1}^m$ is the empty sum.
We split the interval $I_{n-1}$ in two halves of measure ${2^{-n}}$,
given with binary representations of their endpoints as
\begin{align*}
I^{0}_{n} = \left[ 0.d_{1}\ldots d_{n-1}
\ , \ 0.d_{1} \ldots d_{n-1} 1 \right]
\text{ and }
I^{1}_{n}  = \left[ 0.d_{1} \ldots  d_{n-1} 1\ ,  \ 0.d_{1}\ldots
d_{n-1} 111111\ldots \right].
\end{align*}
Since $I^{0}_{n}\cup I^{1}_{n}$ is equal to interval $I_{n-1}$, we have
\[
\mu \left( \Delta _{p_{n}}\cap I^{0}_{n}\right) +
\mu \left(\Delta _{p_{n}}\cap I^{1}_{n}\right) 
=\mu \left( \Delta_{p_{n}}\cap I_{n-1}\right).
\] 
Since $p_{n}\geq p_{n-1}$, we obtain
\[
\mu \left( \Delta_{p_{n}}\cap I^{0}_{n}\right) +
\mu \left( \Delta _{p_{n}}\cap I^{1}_{n}\right) 
\leq 
\mu \left( \Delta _{p_{n-1}}\cap I_{n-1}\right) +r_{p_{n-1}}-r_{p_{n}}.
\] 
Adding
$r_{p_{n}}+r_{p_{n}}$ to both sides of this inequality we obtain
\[
\left( \mu \left( \Delta _{p_{n}}\cap I^{0}_{n}\right) +r_{p_{n}}\right) +
\left( \mu \left( \Delta _{p_{n}}\cap I^{1}_{n}\right) +r_{p_{n}}\right) 
\leq 
\mu \left( \Delta_{p_{n-1}}\cap I_{n-1}\right) +r_{p_{n-1}}+r_{p_{n}}. 
\]
Then, by the inductive condition for $m=n-1$,
\[
\left( \mu \left( \Delta _{p_{n}}\cap I^{0}_{n}\right) + r_{p_{n}}\right) +
\left( \mu \left( \Delta _{p_{n}}\cap I^{1}_{n}\right) + r_{p_{n}}\right) 
<   \frac{1}{2^{n-1}}\Big(\eta +\sum\limits_{j=1}^{n}2^{j-1}\cdot r_{p_{j}}\Big).
\]
Hence, it is impossible that the terms
\[\mu\left( \Delta _{p_{n}}\cap I^{0}_{n}\right) +r_{p_{n}}
\text { and }
\mu \left( \Delta _{p_{n}}\cap I^{1}_{n}\right) +r_{p_{n}}
\]
be both greater than or equal
to
\[
\frac{1}{2^n}\Big( \eta+\sum\limits_{j=1}^{n}2^{j-1}\cdot r_{p_{j}}\Big).
\] 
Let $d \in \{0, 1\}$ be smallest such that 
\[
\mu \left( \Delta _{p_{n}}\cap I_{n}^{d}\right) +r_{p_{n}}
<
\frac{1}{2^n}\Big(\eta
+\sum\limits_{j=1}^{n}2^{j-1}\cdot r_{p_{j}}\Big)
\]
and define  
\[
I_n=I_n^d.
\]
To  verify  that $I_n$ satisfies the inductive condition
it suffices to verify that 
\[
\eta
+\sum\limits_{j=1}^{n}2^{j-1}\cdot r_{p_{j}} <1.
\]
Developing the definition of $r_{p_j}$ we obtain
\begin{align*}
\sum\limits_{j=1}^{n}2^{j-1}\cdot r_{p_{j}} =&\ 
\sum\limits_{j=1}^{n}2^{j-1}\left(
      \sum_{b=2}^{b_{p_j}} \sum_{k=\max(z_b,p_j+1)}^{\infty} \frac{1}{k^2}
    + \sum_{b=b_{p_j}+1}^\infty \sum_{k=z_b}^\infty \frac{1}{k^2}\right)
\\
=&\ 
\left( 
\sum\limits_{j=1}^{n}2^{j-1}\sum_{b=2}^{b_{p_j}}
\sum_{k=\max(z_b,p_j+1)}^{\infty} \frac{1}{k^2}\right)
    +
\left(\sum\limits_{j=1}^{n}  2^{j-1}
 \sum_{b=b_{p_j}+1}^\infty\sum_{k=z_b}^\infty \frac{1}{k^2}\right)
\\
<&\left(\sum\limits_{j=1}^{n}2^{j-1}
      b_{p_j} \sum_{k=p_j+1}^{\infty} \frac{1}{k^2}\right)
    +
\left(\sum\limits_{j=1}^{n}  2^{j-1}
 \sum_{b=b_{p_j}+1}^\infty\frac{\eta}{2^b}\right)
\\
<&\ 
\left(\sum\limits_{j=1}^{n}2^{j-1}       \frac{b_{p_j}}{p_j +1}\right)
+
\left(\sum\limits_{j=1}^{n}  2^{j-1} \frac{\eta}{ 2^{b_{p_j}}} \right)
\\
<&\ 
\left(\sum\limits_{j=1}^{n}2^{j-1}      \frac{2j+2}{2^{2j+2} +1}\right)
+
\left(\sum\limits_{j=1}^{n}  2^{j-1} \frac{\eta }{ 2^{2j+2  } } \right)
\\
<&\  \frac{3}{4}+ \frac{\eta}{4}
\\
< &\  \frac{7}{8}.
\end{align*}
Then, using that $\eta<1/8$ we obtain the desired result,
\[
\mu
\left( \Delta _{p_{n}}\cap I_{n}\right) +r_{p_{n}}
<\frac{1}{2^{n}}\Big( \eta+\sum\limits_{j=1}^{n}2^{j-1}\cdot r_{p_{j}}\Big) 
<\frac{1}{2^{n}} \Big(\eta +\frac{7}{8}\Big)
<\frac{1}{2^{n}}.
\]
\end{proof}

\begin{algorithm}  
\hspace{\dimexpr-\fboxrule-\fboxsep\relax}\fbox{
\begin{minipage}[t]{\textwidth}
{\refstepcounter{theorem}\par\medskip\noindent \textbf{Algorithm \thetheorem\\ } \rmfamily} 

Computation of  the binary expansion $d_1 d_2 \ldots$ of a number $x$
such that for every integer base~$b$,
$\limsup_{N\to \infty} D_N(\{b^j x\}_{j\geq 0})(N/\log \log N)^{1/2}< 3\ C_b$, 
where  $C_b= 166 + 664/(\sqrt{b}-1)$.
\begin{align*}
\omit\rlap{$F(M,N,\alpha_1, \alpha_2, (x_j)_{j\geq 1}) =
\big| \#\{j: M\leq j< M+N:  \alpha_1\leq x_j< \alpha_2 \} -  (\alpha_2-\alpha_1) N \big|,$}
\\
&\phi(N)=\left(2 + 8/(\sqrt{b}-1)\right) \sqrt{N\log\log N},
\\
&T(N)=\lfloor \log N/\log 4\rfloor+1,
\\
&G(b,n, a, h)\!=\!
\{x\in(0,1): F(0,2^n, \alpha_1,\alpha_2,\{b^j x\}_{j\geq 0})\geq  2^{-h/8} \phi(2^n)\},
\\
&\ \ \ \ \text{where }\alpha_1=a 2^{-(h+1)},\  \alpha_2= (a+1) 2^{-(h+1)}, \text{ if } 1\leq h<T(2^{n});
\\
&\ \ \ \ \text{and }  \alpha_1=a 2^{-h},\ \alpha_2= (a+1) 2^{-h}, \text{ if } h=T(2^{n}).
\\
&H(b,n,a, h, \ell, m)\!=\!
\{x\in(0,1): F(2^n+ m2^{\ell},2^{\ell -1},\beta_1,\beta_2,\{b^j x\}_{j\geq 0})\geq  2^{-h/8} 2^{(\ell-n-3)/6}\phi(2^n)\},
\\
&\ \ \ \ \text{where }\beta_1=a 2^{-(h+1)},\ \beta_2= (a+1) 2^{-(h+1)}, \text{ if } 1\leq h<T(2^{\ell-1});
\\
&\ \ \ \ \text{and }  \beta_1=a 2^{-h},\ \beta_2= (a+1) 2^{-h}, \text{ if } h=T(2^{\ell-1}).
\\
&G_{b,n}=
\bigcup_{h=1}^{T(2^n)}
\bigcup_{a=0}^{2^{h}-1} G(b,n,a,h),
 \\
&H_{b,n}=
\bigcup_{h=1}^{T(2^{n})}
\bigcup_{a=0}^{2^{h}-1}
\bigcup_{\ell=n/2}^{n}
\bigcup_{m=1}^{2^{n-\ell}}
H(b,n,a, h, \ell, m).
\end{align*}
{\tt For each base $b$ fix  $z_b\geq 12/\log 2$  such that 
$\sum_{k=z_b}^\infty 1/k^2< 1/(8 \cdot 2^{b})$
\\
$I_{0}=[0,1)$
\\
n=1
\\
repeat
\begin{enumerate}
\item[]  $I^0_n$ is the left half of $I_{n-1}$ and $I^1_n$ is the right half of~$I_{n-1}$
\\
$\begin{array}{l}
\displaystyle
p_n=2^{2n+2}
\\
\displaystyle
b_{p_n}= 2n+2 
\\\displaystyle
\Delta_{p_n}= \bigcup_{b=2}^{b_{p_n}}  \bigcup_{k=z_b}^{p_n} ( G_{b,k}\cup H_{b,k})
\\\displaystyle
r_{p_n}= \sum_{b=2}^{b_{p_n}} \sum_{k=\max(z_b,p_n+1)}^{\infty} k^{-2}
  + \sum_{b=b_{p_n}+1}^\infty \sum_{k=z_b}^\infty k^{-2}
\end{array}$
\item[] if  $\displaystyle \mu(\Delta_{p_n}\cap I^0_n) + r_{p_n}< {2^{-n}}$  then 
\begin{enumerate}
\item[] $d_n=0$
\\ $I_n=I^0_n$
\end{enumerate}
else
\begin{enumerate}
\item[] $d_n=1$
\\ $I_n=I^1_n$
\end{enumerate}
n=n+1
\end{enumerate}
forever

}\end{minipage}}
\end{algorithm}

Let's see that the  number $x=0.d_1 d_2 d_3$ obtained by the next Algorithm \thetheorem \ 
is external to 
\[
\Delta=\bigcup_{b= 2}^\infty\bigcup_{n=n_{z_b}}^\infty (G_{b,n}\cup H_{b,n})
\]

Suppose not. 
Then, there must be an open interval  $J$ in $\Delta $ such that $x \in J$. 
Consider the intervals $I_1^{d_{1}},I_2^{d_{2}},I_3^{d_{3}},\ldots$ 
By our construction, $x $ belongs each of them. 
Let $j$ be the smallest index such that  
$I_j^{d_{j}}\subset J$, which exists because 
the measure of $I_n^{d_{n}}$ goes to $0$ as  $n$ increases. 
Then  $I_j^{d_{j}}$ is fully covered by $\Delta$. 
This  contradicts that in our construction  at each step $n$ we choose  an interval
$I_n^{b_{n}}$  not fully covered by $\Delta $, 
because as  ensured by the proof of Lemma~\ref{lemma:algorithm}, 
\[
\mu(\Delta\cap I_n^{d_n})< 2^{-n}.
\]
We conclude that  $x$ belongs to no interval of $\Delta $.
Recall that we fixed $\delta=1/2$; thus, by Lemma~\ref{lemma:main}, for
for  each integer~$b$ greater than or equal to~$2$, 
\[
\limsup_{N\to \infty }\frac{D_N( \{b^n x\}_{n\geq 0}) \sqrt{N}}{\sqrt{\log \log N}}< 3 C_b,
\]
where $C_b$ is Philipp's constant.

Finally, we count the number of mathematical operations
that the algorithm performs at step $n$ 
to compute the digit $d_n$ in the binary expansion of $x$.
To determine $d_n$,  the algorithm  tests for  $d_n \in \{0,1\}$ whether
\[
\mu(\Delta_{p_n}\cap I^{d_n}_n)+r_{p_n} < 2^n.
\] 
The naive way to obtain this is by constructing the set
$\Delta_{p_n}=\bigcup_{b=2}^{b_{p_n}} \bigcup_{k=z_b}^{p_n} G_{b, k} \cup H_{b,k}$, 
for $b=2, 3, \ldots, b_{p_n}$.
The more demanding is  $G_{b_{p_n},p_n}$ which requires 
 the examination of all the  strings   of digits in  $\{0, \ldots, b_{p_n}-1\}$ of length $2^{p_n}$.
Since $b_{p_n}=2n+2$ and $p_n=2^{2n+2}$,
the number of strings  to be examined is
\[
\sum_{b=2}^{b_{p_n}} b^{2^{p_{n}}} < 2 b_{p_n}^{2^{p_n}} = (2n+2)^{2^{2^{(2n+2)} }}.
\]
Thus, with this naive way, the algorithm  at step $n$  performs 
in the order of  
\[
 (2n+2)^{2^{2^{2n+2} }}
\]
many mathematical operations.

An incremental construction of the sets $G_{b,n} $ and $H_{b,n }$  
can lower the number of needed mathematical operations, 
but would not help to lower the triple-exponential  order of computational complexity.
\bigskip

\noindent
\textbf{Acknowledgements.}  We thank Robert~Tichy for early discussions on this work.  Becher is
supported by the University of Buenos Aires and CONICET.  Becher is a member of the Laboratoire
International Associ\'e INFINIS, CONICET/Universidad de Buenos Aires-CNRS/Universit\'e Paris
Diderot.  Scheerer is supported by the Austrian Science Fund (FWF): I 1751-N26; W1230, Doctoral
Program ``Discrete Mathematics''; and SFB F 5510-N26.  Slaman is partially supported by the
National Science Foundation grant DMS-1600441. The present paper was completed while Becher visited
the Erwin Schr\"odinger International Institute for Mathematics and Physics, Austria.

\bibliographystyle{plain}
\bibliography{ed}

\begin{minipage}{\textwidth}\small
$ $\\
\noindent
Ver\'onica Becher
\\ Departamento de  Computaci\'on,   Facultad de Ciencias Exactas y Naturales
\\Universidad de Buenos Aires \& ICC,  CONICET, Argentina.
\\Pabell\'on I, Ciudad Universitaria, C1428EGA Buenos Aires, Argentina
\\vbecher@dc.uba.ar
\medskip\\
Adrian-Maria Scheerer
\\Institute of Analysis and  Number Theory
\\Graz University of Technology
\\A-8010 Graz, Austria
\\scheerer@math.tugraz.at
\medskip\\
Theodore A. Slaman
\\University of California Berkeley
\\Department of Mathematics
\\719 Evans Hall \#3840, Berkeley, CA 94720-3840 USA
\\slaman@math.berkeley.edu
\end{minipage}
\end{document}